\documentclass[12pt,reqno]{amsart}
\usepackage{amsfonts, amsmath, amssymb, amscd}

\newtheorem{thm}{Theorem}[section]
\newtheorem{cor}[thm]{Corollary}

\newtheorem{lem}[thm]{Lemma}
\newtheorem{pro}[thm]{Proposition}
\theoremstyle{definition}

\theoremstyle{remark}
\newtheorem{rem}[thm]{Remark}
\newcommand{\abs}[1]{\left\vert#1\right\vert}

\newcommand{\ignore}[1]{}
\newcommand{\gn}{\operatorname{ghost\;number}}
\newcommand{\enref}[1]{\ref{enum:#1}}

\newcommand{\heute}{19 February 2015}

\begin{document}

\title[Ghost Number]{On the Christensen-Wang  bounds for the ghost number of a $p$-group algebra}

\author[F.~Altunbulak Aksu]{Fatma Altunbulak Aksu}


\thanks{The first author was supported by  the Scientific and Technical Research Council
of Turkey (T\"UB\.ITAK-BIDEP-2219)}
\address{Dept of Mathematics and Computer Science\\ {\c C}ankaya University\\ 
06790 Ankara \\ Turkey
}
\email{altunbulak@cankaya.edu.tr}

\author[D.~J. Green]{David J. Green}
\address{Dept of Mathematics \\
Friedrich-Schiller-Universit\"at Jena \\ 07737 Jena \\ Germany}
\email{david.green@uni-jena.de}


\date{\heute}

\dedicatory{}

\commby{}

\begin{abstract}
Christensen and Wang give conjectural upper and lower bounds for the ghost number of the group algebra of a $p$-group. We apply results of Koshitani and Motose on the nilpotency index of the Jacobson radical to prove the upper bound and most cases of the lower bound.

\end{abstract}

\maketitle
\begin{section} {Introduction}
Let $G$ be a group and $k$ a field of characteristic $p$. A map $f \colon M \rightarrow N$ in the stable category $\operatorname{stmod}(kG)$ of finitely generated $kG$-modules is called a \emph{ghost} if it vanishes under Tate cohomology, that is if $f_* \colon \hat{H}^*(G,M) \rightarrow \hat{H}^*(G,N)$ is zero. The ghost maps then form an ideal in the stable category; Chebolu, Christensen and Min\'{a}\v{c}~\cite{CheboluChristensenMinac:ghosts} define the \emph{ghost number} of $kG$ to be the nilpotency degree of this ideal.

If $G$ is a $p$-group, then by~\cite{BensonCheboluChristensenMinac:genHyp} the ghost ideal is nontrivial -- that is, the ghost number exceeds one -- unless $G$ is $C_2$ or $C_3$\@.
But the exact value of the ghost number is only known in a few cases; for example, it is not yet known for the quaternion group~$Q_8$.

In \cite{ChristensenWang:ghostNumbers},  Christensen and Wang give conjectural upper and lower bounds for the ghost number of a $p$-group. Our main result establishes most cases of this conjecture:

\begin{thm}
\label{main1}
Let $G$ be a $p$-group of order $p^n$, and $k$ a field of characteristic $p$. Then
\begin{enumerate}
\item \label{enum:main-1}
$\gn(kG) \leq \gn(kC_{p^n})$\@.
\item \label{enum:main-2}
If $G$ is neither extraspecial of exponent~$p$ for odd~$p$, nor extraspecial of order $p^3$ and exponent $p^2$ for $p \in \{3,5\}$, then
\[
\gn(k(C_p)^n) \leq \gn(kG) \, .
\]
\end{enumerate}
\end{thm}

\noindent
We do not know whether the lower bound holds in the excluded cases.
The upper bound is only rarely attained:

\begin{pro}
\label{pro:upperBoundPlus}
Let $G$ be a group of order~$p^n$, and $k$ a field of characteristic~$p$.
If $G$ is non-cyclic but has the same ghost number as $C_{p^n}$, then $p=2$; and $G$ is one of the groups $C_2 \times C_{2^{n-1}}$, $Q_{2^n}$, $\mathit{SD}_{2^n}$ or $\mathit{Mod}_{2^n}$.
\end{pro}

\begin{rem}
By work of Chebolu, Christensen and Min\'{a}\v{c} -- specifically, Theorem~5.4 and Corollary~5.12 of~\cite{CheboluChristensenMinac:ghosts} -- it follows that
\[
\gn(k(C_2 \times C_{2^{n-1}})) = 2^{n-1} = \gn(kC_{2^n}) \, .
\]
We do not know whether the other groups in Proposition~\ref{pro:upperBoundPlus} attain the upper bound.
\end{rem}

\noindent
The \emph{nilpotency index} of the radical $J(kG)$ is the smallest positive integer $s$ such that $J(kG)^s = 0$. Following Wallace~\cite{Wallace:lowerBounds}, we denote the nilpotency index of the radical by $t(G)$. We shall prove Theorem~\ref{main1} using known properties of~$t(G)$. The first link between ghost number and nilpotency index is given by the following result:

\begin{thm} [\cite{CheboluChristensenMinac:ghosts}, Theorem 4.7]
\label{upperbound}
Let $k$ be a field of characteristic $p$ and let $G$ be a finite $p$-group. Then 
\[
\gn(kG) < t(G) \leq \abs{G} \, .
\]
\end{thm}

\noindent
For most $p$-groups we can use $t(G)$ to strengthen the lower bound in Theorem~\ref{main1}\,(\enref{main-2}):

\begin{pro}
\label{main3}
Let $k$ be a field of characteristic $p$. If  $G$ is a $p$-group of order $p^n$ which is not elementary abelian, and moreover
\begin{itemize}
\item
$G$ is neither an extraspecial $2$-group nor an almost extraspecial $2$-group;
\item
$G$ is not extraspecial of exponent~$p$ for $p$ odd;
\item
$G$ is not $p^{1+2}_-$ for $p \in \{3,5\}$;
\item
$G$ is neither $C_4$ nor $C_9$,
\end{itemize}
then \quad $\gn(kG) \geq t((C_p)^n)$.
\end{pro}

\end{section}

\begin{section}{The upper bound}

\noindent
Let us recall the ghost number of a cyclic group.

\begin{thm}[\cite{CheboluChristensenMinac:ghosts}, Theorem 5.4.]
\label{cyclic}
%Let $G$ be a cyclic $p$-group of order $p^n$. Then
\[
\gn(kC_{p^n}) = \left\lceil \frac{p^n-1}{2} \right\rceil = \begin{cases} 2^{n-1} &  p=2  \\ \frac{p^n-1}{2}& \text{$p$ odd} \end{cases} \, .
\]
\end{thm}

\begin{proof}[Proof of Theorem \ref{main1}\,(\enref{main-1})]
Let $G$ be any $p$-group of order~$p^n$.
Theorem~\ref{upperbound} tells us that
\[
\gn(kG) \leq t(G)-1 \, .
\]
Motose and Ninomiya \cite[Thm 1]{MotoseNinomiya:nilpotencyIndex} demonstrated that if $t(G) = \abs{G}$ then $G$ is cyclic; and Koshitani \cite[Thm 1.6]{Koshitani:nilpotencyIndices} showed that if $n \geq 2$ then the following three statements are equivalent:
\begin{enumerate}
\item $t(G) = p^{n-1}+p-1$
\item $p^{n-1} < t(G) < p^n$
\item $G$ is not cyclic, but it does have a cyclic subgroup of index~$p$. 
\end{enumerate}
If $p=2$ and $G$ is not cyclic then by Koshitani's result and Theorem~\ref{cyclic}
\[
t(G) -1 \leq 2^{n-1} = \gn(kC_{p^n}) \, .
\]
If $p$ is odd and $G$ not cyclic, then $t(G)-1 \leq p^{n-1}+p-2$. This is strictly smaller than $\gn(kC_{p^n}) = \frac{p^n-1}{2}$, except for the one case $p^n=3^2$. But the cyclic group of order $9$ has ghost number~$4$, whereas $C_3 \times C_3$ has ghost number~$3$ by \cite[Thm 1.1]{ChristensenWang:ghostNumbers}\@.
\end{proof}

\begin{proof}[Proof of Proposition~\ref{pro:upperBoundPlus}]
Inspecting the proof of Theorem \ref{main1}\,(\enref{main-1}) we see that $p=2$, and that $G$ has a cyclic subgroup of index~$p$. By the classification of such groups (see e.g.\@ \cite[23.4]{Aschbacher:book}) it follows that $G$ is either $D_{2^n}$ or one of the stated groups. But $D_{2^n}$ has ghost number $2^{n-2}+1$ by \cite[Cor 1.1]{ChristensenWang:ghostNumbers}\@.
\end{proof}

\end{section}

\begin{section}{Nilpotency index and a lower bound}

\noindent
The following result is a special case of \cite[Thm 4.3]{ChristensenWang:ghostNumbers}:

\begin{thm}[Christensen--Wang]
\label{thm:CW-lower}
Let $G$ be a finite $p$-group and $k$ a field of characteristic~$p$. Suppose that $C \leq Z(G)$ is cyclic of order~$p$. Then
\[
\gn(kG) \geq t(G/C) \, .
\]
\end{thm}

\begin{proof}
In \cite[Thm 4.3]{ChristensenWang:ghostNumbers}, take $M_n$ to be the trivial $kC$-module. Then the induced $kG$-module $k(G/C)$ has ghost length equal to its radical length. But its radical length is $t(G/C)$, and by definition the ghost number is the largest ghost length.
\end{proof}

\noindent One immediate corollary generalizes \cite[Corollary~5.12]{CheboluChristensenMinac:ghosts}:

\begin{cor}
\label{cor:directproduct}
Let $H$ be a $2$-group and $G=H\times (C_2)^r$ for $r\geq 1$.
Then
\[
\gn(kG) = t(G) - 1 = t(H)+r-1 \, .
\]
\end{cor}

\begin{proof}
The  Jennings series of $G$ is given by
\[
\Gamma_s(G) = \begin{cases} \Gamma_1(H) \times C_2^r & s=1 \\ \Gamma_s(H) & \text{otherwise} \end{cases} \, .
\]
By Jennings' Theorem (Theorem 3.14.6 in~\cite{Benson:I}) it follows that $t(G) = t(H)+r$. For the first inequality it suffices to consider the case $r=1$; so $G = H \times C$ with $C \cong C_2$. Theorem~\ref{thm:CW-lower} tells us that
\[
\gn(kG) \geq t(G/C) = t(H) = t(G)-1 \, .
\]
Now apply Theorem \ref{upperbound}\@.
\end{proof}

\end{section}

\begin{section}{Proposition~\ref{main3}: the (almost) extraspecial case}

\noindent
Recall that a $p$-group $G$ is \emph{extraspecial} if $\Phi(G) $, $[G,G]$ and $Z(G)$ coincide and have order~$p$; and \emph{almost} extraspecial if $\Phi(G) = [G,G]$ has order~$p$, but $Z(G)$ is cyclic of order~$p^2$. That is, an almost extraspecial group is a central product of the form $H * C_{p^2}$, with $H$ extraspecial.
The following lemma is presumably well known.

\begin{lem}
\label{lem:extrasp}
Suppose that $G$ is a $p$-group of order $p^n$ whose Frattini subgroup has order~$p$. Then
\[
t(G) = \begin{cases} (n+1)(p-1)+1 & \text{if $G$ has exponent~$p$} \\ (p+n-1)(p-1) + 1 & \text{if $G$ has exponent $p^2$} \end{cases} \, .
\]
In particular, if $p=2$ then $t(G) = n+2$.
\end{lem}

\begin{proof}
Consider the Jennings series $\Gamma_r(G)$. We certainly have $\Gamma_1(G) = G$ and $\Gamma_2(G) = \Phi(G)$.
If the exponent is~$p$ then $\Gamma_3(G) = 1$, whereas if the exponent is~$p^2$ then $\Gamma_p(G) = \Gamma_2(G)$ and $\Gamma_{p+1}(G)=1$. The result follows by Jennings' Theorem (Theorem 3.14.6 in~\cite{Benson:I})\@.
\end{proof}

\noindent
Recall from \cite[Thm 1]{MotoseNinomiya:nilpotencyIndex} that $t(C_p^n) = n(p-1)+1$.

\begin{pro}
\label{evenlbound}
Let $G$ be a $2$-group of order $2^n$ whose Frattini subgroup has order~$2$. If $G$ is neither $C_4$ nor extraspecial nor almost extraspecial then
\[
\gn(kG) = n+1 = t(C_2^n) \, .
\]
\end{pro}

\begin{proof}
By assumption, $G$ has the form $G = H \times C$ with $C \cong C_2$, and $\Phi(H)$ cyclic of order~$2$. So $\gn(kG) = t(H)$ by Corollary~\ref{cor:directproduct}, and $t(H) = n+1$ by Lemma~\ref{lem:extrasp}\@.
\end{proof}

\begin{pro}\label{poddlbound}
Let $p$ be an odd prime and $G$ be a $p$-group of order $p^n$ whose Frattini subgroup has order~$p$. Then
\[
\gn(kG) \geq n(p-1)+1 = t(C_p^n)
\]
except possibly in the following cases:
\begin{itemize}
\item
$G$ is extraspecial of exponent~$p$, for any odd~$p$;
\item
$G$ is extraspecial of order~$p^3$ and exponent $p^2$ for $p \in \{3,5\}$.
\item $G = C_9$, with ghost number $4$ and $t(C_3 \times C_3) = 5$.
\end{itemize}
\end{pro}

\begin{rem}
\label{rem:subgp}
In the proof we use the Proposition~4.9 from \cite{CheboluChristensenMinac:ghosts}: If $H$ is a subgroup of a finite $p$-group~$G$, then
\[
\gn(kH) \leq \gn(kG) \, .
\]
\end{rem}

\begin{proof}
By assumption we have $\Phi(G) \leq \Omega_1(Z(G))$. Since $\Phi(G) \neq 1$ we have $n \geq 2$.
\item \textbf{Step 1}: Reduction to the case $\Phi(G) = [G,G] = \Omega_1(Z(G))$. \\
If $\Phi(G) \lneq \Omega_1(Z(G))$ then there is $C \leq Z(G)$ with $\abs{C} = p$ and $C \cap \Phi(G) = 1$, hence $\abs{\Phi(G/C)} = p$ and so by Theorem~\ref{thm:CW-lower} and Lemma~\ref{lem:extrasp}
\[
\gn(G) \geq t(G/C) \geq n(p-1) +1 \, .
\]
So we may assume that $\Phi(G) = \Omega_1(Z(G))$. If $[G,G] \neq \Phi(G)$ then $G$ is abelian; and therefore cyclic of order~$p^2$, since $\Omega_1(G) = \Phi(G)$. By Theorem~\ref{cyclic} the ghost number is $\frac{p^2-1}{2}$; for $p > 3$ this is at least $2p-1$.
\item \textbf{Step 2}: Reduction to the case $G$ extraspecial. \\
Extraspecial means that $\Phi(G) = Z(G) = [G,G]$. So if $G$ is not extraspecial then $\Phi(G) \lneq Z(G)$, so $Z(G) \cong C_{p^2}$ and there is a maximal subgroup $E < G$ with $G = EZ(G)$ and $E \cap Z(G) = \Phi(G)$. It follows that $E$ is extraspecial, with $\Phi(E) = \Phi(G)$. As $E$ is extraspecial it has a maximal subgroup of the form $H \times C_p$, with $\Phi(G) \leq H$. Then $L := HZ(G) \times C_p$ is maximal in $G$, and by Theorem~\ref{thm:CW-lower} and Remark~\ref{rem:subgp}
\[
\gn(kG) \geq \gn(L) \geq t(HZ(G)) \, .
\]
As $HZ(G)$ has order $p^{n-2}$ and exponent~$p^2$, Lemma~\ref{lem:extrasp} says that $t(HZ(G)) = (p+n-3)(p-1)+1$.
\item \textbf{Step 3}: Reduction to the case $G \cong p^{1+2}_-$ \\
We may asume that $G$ has exponent~$p^2$, so $G \cong p^{1+2r}_+$. If $r \geq 2$ then $G$ has a maximal subgroup of the form $H \times C_p$, where $H$ has order $p^{n-2}$ and exponent $p^2$. The result now follows by the argument of the previous step.
\item \textbf{Step 4}: The case $G \cong p^{1+2}_-$ \\
$G$ has a subgroup of the form $C_{p^2}$,
so $\gn(kG) \geq \gn(C_{p^2})$ by Remark~\ref{rem:subgp}\@. But $C_{p^2}$ has ghost number $\frac{p^2-1}{2}$, which exceeds $3p-2$ for $p \geq 7$.
% so by Remark~\ref{rem:subgp} and Theorem~\ref{cyclic}
% \[
% \gn(kG) \geq \gn(C_{p^2}) = \frac{p^2-1}{2} \, .
% \]
% If $p \geq 7$  then $\frac{p^2-1}{2} \geq 3p-2$.
\end{proof}

\end{section}

\begin{section}{The lower bound}

\begin{lem}
\label{notelab}
Let G be a $p$-group of order $p^n$. If $\abs{\Phi(G)} > p$, then
\[
\gn(kG) > n(p-1) + 1 = t(C_p^n) \, .
\]
\end{lem}

\begin{proof}
Let $C \leq \Phi(G) \cap Z(G)$ be cyclic of order~$p$. Then $\gn(kG) \geq t(G/C)$ by Theorem~\ref{thm:CW-lower}. Since $G/C$ has order $p^{n-1}$ and is not elementary abelian, we have $t(G/C) \geq n(p-1)+1$ by \cite[Thm 6]{Motose:Koshitani}\@.
But $t(C_p^n) = n(p-1)+1$ by \cite[Thm 1]{MotoseNinomiya:nilpotencyIndex}\@.
\end{proof}

\begin{proof}[Proof of Proposition~\ref{main3}]
Follows from Propositions \ref{evenlbound}~and \ref{poddlbound}, and Lemma~\ref{notelab}\@.
\end{proof}

\begin{proof}[Proof of Theorem \ref{main1}\,(\enref{main-2})]
Let $G$ have order~$p^n$.

First suppose that $p=2$. Let $C \leq Z(G)$ have order~$2$, then $\gn(kG) \geq t(G/C)$ by Theorem~\ref{thm:CW-lower}. Since $G/C$ has order $2^{n-1}$ we have $t(G/C) \geq (n-1)+1 = n = t(C_2^n) - 1$. The result follows since $\gn(kH) \leq t(H)-1$ by Theorem~\ref{upperbound}\@.

Now suppose that $p$ is odd.
By Proposition~\ref{main3} we only need consider the case $G = C_9$, with ghost number~$4$. But $C_3 \times C_3$ has ghost number~$3$ by \cite[Thm 1.1]{ChristensenWang:ghostNumbers}\@.
\end{proof}

\end{section}

% Papers on nilpotency index:
% \cite{Jennings:groupRing},
% \cite{MotoseNinomiya:nilpotencyIndex},
% \cite{Koshitani:nilpotencyIndices}, 
% \cite{Motose:Koshitani}\@.
% Benson vol~1: \cite{Benson:I}.
%Ghosts: \cite{BensonCheboluChristensenMinac:genHyp}, \cite{CheboluChristensenMinac:ghosts}, \cite{ChristensenWang:ghostNumbers}, \cite{ChristensenWang:ghostNumbers2}, \cite{Freyd:generating}, \cite{Stancu:resistant}.
% Probably don't need Wallace:~\cite{Wallace:lowerBounds}\@.
% Aschbacher: \cite{Aschbacher:book}

% \bibliographystyle{abbrv}
% % \bibliographystyle{amsplain}
% \bibliography{united}

\begin{thebibliography}{1}

\bibitem{Aschbacher:book}
M.~Aschbacher.
\newblock {\em Finite group theory}, volume~10 of {\em Cambridge Studies in
  Advanced Mathematics}.
\newblock Cambridge University Press, Cambridge, 1986.

\bibitem{Benson:I}
D.~J. Benson.
\newblock {\em Representations and Cohomology. {I}}.
\newblock Cambridge Studies in Advanced Math., vol.~30. Cambridge University
  Press, Cambridge, second edition, 1998.

\bibitem{BensonCheboluChristensenMinac:genHyp}
D.~J. Benson, S.~K. Chebolu, J.~D. Christensen, and J.~Min{\'a}{\v{c}}.
\newblock The generating hypothesis for the stable module category of a
  {$p$}-group.
\newblock {\em J. Algebra}, 310(1):428--433, 2007.

\bibitem{CheboluChristensenMinac:ghosts}
S.~K. Chebolu, J.~D. Christensen, and J.~Min{\'a}{\v{c}}.
\newblock Ghosts in modular representation theory.
\newblock {\em Adv. Math.}, 217(6):2782--2799, 2008.

\bibitem{ChristensenWang:ghostNumbers}
J.~D. Christensen and G.~Wang.
\newblock Ghost numbers of group algebras.
\newblock {\em Algebras and Representation Theory}, pages 1--33, 2014.

\bibitem{Koshitani:nilpotencyIndices}
S.~Koshitani.
\newblock On the nilpotency indices of the radicals of group algebras of
  {$P$}-groups which have cyclic subgroups of index {$P$}.
\newblock {\em Tsukuba J. Math.}, 1:137--148, 1977.

\bibitem{Motose:Koshitani}
K.~Motose.
\newblock On a theorem of {S}. {K}oshitani.
\newblock {\em Math. J. Okayama Univ.}, 20(1):59--65, 1978.

\bibitem{MotoseNinomiya:nilpotencyIndex}
K.~Motose and Y.~Ninomiya.
\newblock On the nilpotency index of the radical of a group algebra.
\newblock {\em Hokkaido Math. J.}, 4(2):261--264, 1975.

\bibitem{Wallace:lowerBounds}
D.~A.~R. Wallace.
\newblock Lower bounds for the radical of the group algebra of a finite
  {$p$}-soluble group.
\newblock {\em Proc. Edinburgh Math. Soc. (2)}, 16:127--134, 1968/1969.

\end{thebibliography}

\end{document}